\documentclass[12pt,leqno,fleqn]{amsart}
\usepackage{amsmath,amstext,amsthm,amssymb,amsxtra}
\usepackage[top=1.5in, bottom=1.5in, left=1.25in, right=1.25in]	{geometry}
\usepackage[normalem]{ulem}
\usepackage{txfonts} 
\usepackage[T1]{fontenc}
\usepackage{lmodern}
\usepackage{tikz}\usetikzlibrary{arrows}

\usepackage{euler}   

\makeatletter
\DeclareRobustCommand\widecheck[1]{{\mathpalette\@widecheck{#1}}}
\def\@widecheck#1#2{%
    \setbox\z@\hbox{\m@th$#1#2$}%
    \setbox\tw@\hbox{\m@th$#1%
       \widehat{%
          \vrule\@width\z@\@height\ht\z@
          \vrule\@height\z@\@width\wd\z@}$}%
    \dp\tw@-\ht\z@
    \@tempdima\ht\z@ \advance\@tempdima2\ht\tw@ \divide\@tempdima\thr@@
    \setbox\tw@\hbox{%
       \raise\@tempdima\hbox{\scalebox{1}[-1]{\lower\@tempdima\box
\tw@}}}%
    {\ooalign{\box\tw@ \cr \box\z@}}}
\makeatother

\usepackage{mathtools}
\mathtoolsset{showonlyrefs,showmanualtags}

\usepackage{hyperref} 
\hypersetup{
    colorlinks=true,       
    linkcolor=blue,          
    citecolor=magenta,        
    filecolor=magenta,      
    urlcolor=cyan           
}

\usepackage[msc-links]{amsrefs}

\theoremstyle{plain} 
\newtheorem{lemma}[equation]{Lemma}

\newtheorem{theorem}[equation]{Theorem}

\theoremstyle{definition}
\newtheorem{definition}[equation]{Definition}

\theoremstyle{remark}
\newtheorem{remark}[equation]{Remark}

\numberwithin{equation}{section}

\title[Boundedness of Doob Maximal Operators]{A Note on the Boundedness of Doob Maximal Operators on a Filtered Measure Space}

\author[W. Chen]{Wei Chen}
\address{ School of Mathematical Sciences, Yangzhou University, Yangzhou 225002, China}
\email {weichen@yzu.edu.cn}

\author[J. Y. Cui]{Jingya Cui}
\address{ School of Mathematical Sciences, Yangzhou University, Yangzhou 225002, China}
\email {jycui$\_$yzu@163.com}

\thanks{The research of W. Chen is supported by the National Natural Science Foundation of China(11971419, 11771379) and the School Foundation of Yangzhou University (2019CXJ001).}

\begin{document}
\begin{abstract}
Let $M$ be the Doob maximal operator on a filtered measure space and let $v$ be an $A_p$ weight with $1<p<+\infty$. 
We try proving that \begin{equation}\lVert M  f\rVert _{L ^{p}(v) }\leq p^{\prime}[v]^{\frac{1}{p-1}}_{A_p}\lVert  f\rVert _{L ^{p} (v)},\end{equation}
where $1/p+1/p^{\prime}=1.$
Although we do not find an approach which gives the constant $p^{\prime},$  we obtain that
\begin{equation}\lVert M  f\rVert _{L ^{p}(v) }\leq p^{\frac{1}{p-1}}p^{\prime}[v]^{\frac{1}{p-1}}_{A_p}\lVert  f\rVert _{L ^{p} (v)},
\end{equation}
with
$\lim\limits_{p\rightarrow+\infty}p^{\frac{1}{p-1}}=1.$

\end{abstract}

	\maketitle
\tableofcontents

\section{Introduction}\label{Pre}

Let $M$ be the Doob maximal operator on a filtered measure space. For $1<p<+\infty,$ it is well known (see e.g. \cite{MR1224450}) that 
\begin{equation}\label{ineq-n}
\lVert M  f\rVert _{L ^{p} }\leq p’\lVert  f\rVert _{L ^{p} },
\end{equation} where $1/p+1/p’=1$ and $p’$ is the best constant.
Let $v$ be an $A_p$ weight with $1<p<+\infty.$  Tanaka and Terasawa \cite{MR3004953}  proved that
\begin{equation}\label{ineq-w}
\lVert M  f\rVert _{L ^{p}(v) }\leq C[v]^{\frac{1}{p-1}}_{A_p}\lVert  f\rVert _{L ^{p} (v)},
\end{equation}
where $C$ is independent of $v.$

For a Euclidean space with a dyadic filtration, the dyadic maximal operator is the above Doob maximal operator. For  the dyadic maximal operator, the constant $1/(p-1)$ is the optimal power on $[v]_{A_{p}}$(see e.g. \cite{MR2534183} or \cite{ MR2399047}) . It follows that the constant $1/(p-1)$ is also the optimal power on $[v]_{A_{p}}$ for the Doob maximal operator $M.$ 

In this note, we estimate the constant $C$ in \eqref{ineq-w}. Substituting $v=1$ into \eqref{ineq-w},  we  get \eqref{ineq-n}.  Thus, we conjecture that the constant $C$ equals $p’$ in \eqref{ineq-w}. 
But we do not find an approach which gives the constant $C=p’.$  Our results are as follows.

\begin{theorem}\label{thm_Ap} Let $v$ be a weight  and $1< p< \infty.$ 
We have the inequality
\begin{equation}\label{norm}
\lVert M  f\rVert _{L ^{p} (v)} \leq C \lVert f\rVert _{L ^{p} (v)}
\end{equation}
if and only if $v\in  A_p$.
Moreover, if we denote the smallest constant in \eqref{norm} by $\|M\|$, we have
\begin{equation}\label{Ap_con}
[v]_{A_p}\leq\|M\|^p
\end{equation}
and
\begin{equation}\label{M_con}
\|M\|\leq p^{\frac{1}{p-1}}p’[v]^{\frac{1}{p-1}}_{A_p}.
\end{equation}
\end{theorem}

\begin{remark}\label{rkapp}The content of Theorem \ref{thm_Ap}  is \eqref{M_con}.
In order to prove \eqref{M_con}, we use different approaches as follows:
\begin{enumerate}
  \item \label{Lernerpf} Motivated by the proof of \cite[Theorem B]{MR2399047}, we get $C=p^{\frac{1}{p-1}}p’.$ 
  \item \label{Chenpri} Using the construction of principal sets \cite{MR3004953} and the conditional sparsity \cite{MR4125846}, we have $C=a^2\eta^{(p’-1)}p’,$
where $a,~\eta$ are the constants in the construction of principal sets (Appendix \ref{Appe}).
 \item \label{Long} Long \cite[Theorem 6.6.3]{MR1224450} qualitatively evaluated $\|M\|$. Modifying Long’s proof, we have $C=p^{\frac{1}{p-1}}p’$ which is the same as \eqref{Lernerpf}. 
\end{enumerate}
Approaches \eqref{Lernerpf} and \eqref{Long} both use the boundedness of Doob maximal operator twice and give the same estimation $C=p^{\frac{1}{p-1}}p’.$  
Approach \eqref{Chenpri}  depends on the conditional sparsity and the boundedness of Doob maximal operator. Letting $\sigma=v^{\frac{1}{p-1}}$
and $f=h\sigma,$ we can rewrite \eqref{norm} as
\begin{equation}\label{norm2}
\lVert M  (h\sigma)\rVert _{L ^{p} (v)} \leq C \lVert h\sigma\rVert _{L ^{p} (\sigma)}.
\end{equation}
Cao and Xue \cite{MR3424618} (see also the references therein)  used the 
atomic decomposition to study weighted theory on the Euclidean space, but we do not know whether it is possible on the filtered measure space.
\end{remark}

This paper is organized as follows. Sect. \ref{pre1} consists of the preliminaries for this paper. In Sect. \ref{Proofthm3} we give 
the proof of Theorem \ref{thm_Ap} , and in Sect. \ref{Compar} we compare $p^{\frac{1}{p-1}}$ with $a^2\eta^{(p’-1)}$. 
In order to keep track the constants in our paper, we modify the construction of principal sets in Appendix \ref{Appe}.

\section{Preliminaries} \label{pre1}
The filtered measure space was discussed in \cite{MR3004953,MR3617205}, which is abstract and contains several kinds of spaces. For example,  a doubling metric space with systems of dyadic cubes was introduced in Hyt\"{o}nen and Kairema \cite{MR2901199}. 
In order to develop discrete martingale theory, a probability space endowed with a family of $\sigma$-algebra was considered in Long \cite{MR1224450}. In addition,  a Euclidean space with several adjacent systems of dyadic cubes
was mentioned in Hyt\"{o}nen \cite{MR2912709}. Because the filtered measure space  is abstract, it is possible to study these spaces together(\cite{MR2854692, MR3644418, MR1267569}).
As is well known, Lacey,
Petermichl and Reguera \cite{MR2657437}  studied the shift  operators, which is related to the martingale theory on a filtered measure space.
When Hyt\"{o}nen \cite{MR2912709} solved the conjecture of $A_2,$  those operators are very useful. 

\subsection{Filtered Measure Space} 
Let $(\Omega,\mathcal{F},\mu)$ be a measure space and  let  $\mathcal{F}^0=\bigcup\{E:E\in \mathcal{F},~\mu(E)<+\infty\}$. 
 As for $\sigma$-finite, we mean that $\Omega$ is a union of $(E_i)_{i\in \mathbb Z}\subset\mathcal{F}^0.$ We only consider $\sigma$-finite measure space $(\Omega,\mathcal{F},\mu)$ in this paper.
Let $\mathcal{B}$ be a sub-family of $\mathcal{F}^0$ and let $f:\Omega\rightarrow \mathbb R$ be measurable on $(\Omega,\mathcal{F},\mu)$. If for all $B\in \mathcal{B}$, we have 
$\int_B|f|d\mu<+\infty,$ then we say that $f$ is $\mathcal{B}$-integrable. The family of the above functions is denote by $L^1_{\mathcal{B}}(\mathcal{F},\mu).$

Let $\mathcal{B}\subset\mathcal{F}$ be a sub-$\sigma$-algebra and let $f\in L^1_{\mathcal{B}^0}(\mathcal{F},\mu)$. 
Because of $\sigma$-finiteness of $(\Omega,\mathcal{B},\mu)$ and Radon-Nikod\'ym’s theorem, there is a unique function denoted by
$\mathbb E(f|\mathcal{B})\in L^1_{\mathcal{B}^0}(\mathcal{B},\mu)$ or $\mathbb E_{\mathcal{B}}(f)\in L^1_{\mathcal{B}^0}(\mathcal{B},\mu)$ 
such that
\begin{equation*}
\int_Bfd\mu=\int_B \mathbb E_{\mathcal{B}}(f) d\mu,\quad \forall B\in\mathcal{B}^0.
\end{equation*}

Letting $(\Omega,\mathcal{F},\mu)$ with a family $(\mathcal{F}_i)_{i\in \mathbb Z}$ of sub-$\sigma$-algebras satisfying that $(\mathcal{F}_i)_{i\in \mathbb Z}$ is increasing, we say that 
$\mathcal{F}$  has a filtration $(\mathcal{F}_i)_{i\in \mathbb Z}$. Then, a quadruplet $(\Omega,\mathcal{F},\mu; (\mathcal{F}_i)_{i\in \mathbb Z})$
is said to be a filtered measure space.
It is clear that $L^1_{\mathcal{F}^0_i}(\mathcal{F},\mu)\supset L^1_{\mathcal{F}^0_j}(\mathcal{F},\mu)$ with $i < j.$ Let
$\mathcal{L}:=\bigcap\limits_{i\in \mathbb Z}L^1_{\mathcal{F}^0_i}(\mathcal{F},\mu)$ and $f\in\mathcal{L},$ then $(\mathbb E_i(f))_{i\in\mathbb Z}$ is a martingale, where 
$\mathbb E_i(f)$ means $\mathbb E(f|\mathcal{F}_i).$ The reason is that $\mathbb E_i(f)=\mathbb E_i(\mathbb E_{i+1}(f)),$ $i\in \mathbb Z.$

\subsection{Stopping Times}\label{stopping}

Let $(\Omega,\mathcal{F},\mu; (\mathcal{F}_i)_{i\in \mathbb Z})$  be a $\sigma$-finite filtered measure space and let $\tau:~ \Omega\rightarrow\{-\infty\}\cup \mathbb Z\cup\{+\infty\}$. If for any $i\in \mathbb Z,$ we have
$\{\tau=i\}\in \mathcal{F}_i,$ then $\tau$ is said to be a stopping time. We denote the family of all stopping times by $\mathcal{T}.$
For $i\in \mathbb Z,$ we denote $\mathcal{T}_i:=\{\tau\in \mathcal{T}:~\tau\geq i\}.$

\subsection{Operators and Weights}
Let $f\in\mathcal{L}.$ The Doob maximal operator is defined by
$$Mf=\sup_{i\in \mathbb Z}|\mathbb E_i(f)|.$$
For $i\in \mathbb Z,$ we define the tailed Doob maximal operator
by $${^*M}_if=\sup_{j\geq i}|\mathbb E_j(f)|.$$

For $\omega\in \mathcal{L}$ with $\omega\geq 0,$ we say that $\omega$ is a weight. The set of all weights is denoted by $\mathcal{L}^+.$
Let $B\in \mathcal {F},$ $\omega\in \mathcal{L}^+.$
Then $\int_\Omega\chi_Bd\mu$ and
$\int_\Omega\chi_B\omega d\mu$  are denoted by $|B|$ and $|B|_\omega,$
respectively. Now we give the definition of $A_p$ weights.

\begin{definition}\label{definition o Ap}Let $1<p<\infty$ and let $\omega$ be a weight. 
We say that the weight $\omega$
is  an $A_p$ weight, if
there exists a positive constant $C$ such that
      \begin{equation}\label{AP1}
\sup\limits_{j\in \mathbb Z}\mathbb E_j(\omega)\mathbb E_j(\omega^{1-p'})^{\frac{p}{p'}}\leq C,
      \end{equation}
where $\frac{1}{p}+\frac{1}{p'}=1.$ We denote
the smallest constant $C$ in \eqref{AP1} by $[\omega]_{A_p}$.
\end{definition}

\section{Approaches of Theorem \ref{thm_Ap}}
\label{Proofthm3}

\begin{proof}

We prove that \eqref{norm} implies \eqref{Ap_con}. For $i\in \mathbb{Z}$ and $B\in \mathcal{F}^0_i,$ we let
$f=\chi_B.$
Then $$\mathbb E_i(v^{-\frac{1}{p-1}})\chi_B
\leq M(f\sigma)\chi_B,$$
where $\sigma=v^{\frac{1}{p-1}}.$
It follows from \eqref{norm} that
\begin{eqnarray*}
\Big(\int_B\mathbb E_i(v^{-\frac{1}{p-1}})^pvd\mu\Big)^{\frac{1}{p}}
\leq \|M\|\Big(\int_\Omega v^{-\frac{1}{p-1}}\chi_Bd\mu\Big)^{\frac{1}{p}}.
\end{eqnarray*}

Thus
\begin{eqnarray*}
\mathbb E_i(v^{-\frac{1}{p-1}})^p\mathbb E_i(v)\leq \|M\|^p\mathbb E_i(v^{-\frac{1}{p-1}}),
\end{eqnarray*}
which shows that $$[v]_{A_p}\leq\|M\|^p.$$

In order to prove \eqref{M_con}, we provide  the three approaches which we mentioned in Remark \ref{rkapp}.

{\bf Approach \eqref{Lernerpf}.}
It is clear that
\begin{eqnarray}
\mathbb E_n(f)
&=&\Big(\mathbb E_n(v)\mathbb E_n(\sigma)^{p-1}\frac{1}{\mathbb E_n(v)}\big(\frac{1}{\mathbb E_n(\sigma)}\mathbb E_n(f)\big)^{p-1}\Big)^{\frac{1}{p-1}}\\
&=&\Big(\mathbb E_n(v)\mathbb E_n(\sigma)^{p-1}\Big)^{\frac{1}{p-1}} \Big(\frac{1}{\mathbb E_n(v)}\big(\frac{1}{\mathbb E_n(\sigma)}\mathbb E_n(f)\big)^{p-1}\Big)^{\frac{1}{p-1}}\\
&\leq&[v]^{\frac{1}{p-1}}_{A_{p}}M^{v}\big(v^{-1}M^{\sigma}(f\sigma^{-1})^{p-1}\big)^{\frac{1}{p-1}}.
\end{eqnarray}
Then we have 
\begin{eqnarray}
M(f)\leq[v]^{\frac{1}{p-1}}_{A_{p}}M^{v}\big(v^{-1}M^{\sigma}(f\sigma^{-1})^{p-1}\big)^{\frac{1}{p-1}}.
\end{eqnarray}
Using the boundedness of Doob maximal operators $M^v$ and $M^{\sigma},$ we obtain
\begin{eqnarray}
\|M(f)\|_{L^p(v)}
&\leq&[v]^{\frac{1}{p-1}}_{A_{p}}\|M^{v}\big(v^{-1}M^{\sigma}(f\sigma^{-1})^{p-1}\big)^{\frac{1}{p-1}}\|_{L^p(v)}\\
&=&[v]^{\frac{1}{p-1}}_{A_{p}}\|M^{v}\big(v^{-1}M^{\sigma}(f\sigma^{-1})^{p-1}\big)\|^{\frac{1}{p-1}}_{L^{p’}(v)}\\
&\leq&p^{\frac{1}{p-1}}[v]^{\frac{1}{p-1}}_{A_{p}}\|M^{\sigma}(f\sigma^{-1})\|_{L^p(\sigma)}\\
&\leq&p^{\frac{1}{p-1}}p’[v]^{\frac{1}{p-1}}_{A_{p}}\|f\|_{L^p(v)}\label{wqw}.
\end{eqnarray}

{\bf Approach \eqref{Chenpri}.} For $i\in \mathbb Z,$ $k\in \mathbb Z$ and $\Omega_0\in \mathcal{F}^0_i,$ we denote $$
P_0=\{a^{k-1}< \mathbb E(f\sigma|\mathcal{F}_i)\leq a^k\}\cap\Omega_0.$$
We claim that
\begin{multline}\label{6}
\Big(\int_{P_0}{^*M_i}
(f\sigma\chi_{P_0})^{p}vd\mu\Big)^{\frac{1}{p}} \leq
a^2\eta^{(p’-1)}p’[v]^{\frac{p’}{p}}_{A_{p}}\Big(\int_{P_0} f^p\sigma d\mu\Big)^{\frac{1}{p}},
\end{multline}
where $a,~\eta$ are the constants in the construction of principal sets (Appendix \ref{Appe}).
To see this, denote $h=f\sigma\chi_{P_0}.$ 
For the above $i,$ $P_0$ and $h,$ we construct principal sets. Then, Lemma $\ref{repre}$ shows that
\begin{equation}\label{doob}
\int_{P_0}{^*M_i}
(f\sigma)^{p}vd\mu
\leq a^{2p}\sum\limits_{P\in \mathcal{P}}\int_{E(P)}a^{p({\mathcal{K}}_2(P)-1)}vd\mu.
\end{equation}

To estimate $|E(P)|_v.$ For the sake of simplicity, we denote $E_{\mathcal{F}_{{\mathcal{K}}_1(P)}}(\cdot)$ by 
$E_P(\cdot)$ without confusion.
We now estimate $|E(P)|_v$ as follows:

\begin{eqnarray*}
|E(P)|_v&\leq&|P|_v=\int_{P}\mathbb E_P(v)d\mu\\
&=&\int_{P}\mathbb E_P(v)^{p’}\mathbb E_P(v)^{1-p’}
\mathbb E_P(\sigma)^{p}\mathbb E_P(\sigma)^{-p}d\mu\\
&=&\int_{P}\mathbb E_P(v)^{p’}
\mathbb E_P(\sigma)^{p}\mathbb E_P(v)^{1-p’}\mathbb E_P(\sigma)^{-p}d\mu.
\end{eqnarray*}
In the view of the definition of $A_p$ and the construction of $\mathcal{P},$ we have
\begin{eqnarray*}
|E(P)|_v&\leq&[v]^{p’}_{A_{p}}\int_{P}\mathbb E_P(v)^{1-p'}
\mathbb E_P(\sigma)^{-p}d\mu\\
&\leq&\eta^{p(p’-1)}[v]^{p’}_{A_{p}}\int_{P}\mathbb E_P(v)^{1-p’}
    \mathbb E_P(\sigma)^{-p}
    \mathbb E_P(\chi_{E(P)})^{p(p'-1)}d\mu\\
&=&\eta^{p(p’-1)}[v]^{p’}_{A_{p}}\int_{P}\mathbb E_P(v)^{1-p’}
    \mathbb E_P(\sigma)^{-p}\mathbb E_P(\chi_{E(P)}v^{\frac{1}{p}}\sigma^{\frac{1}{p'}})^{p(p'-1)}d\mu.
\end{eqnarray*}
Noting that the conditional expectation satisfies H\"{o}lder's inequality,  we have
\begin{eqnarray*} |E(P)|_v
&\leq&\eta^{p(p’-1)}[v]^{p’}_{A_{p}}\int_{P}\mathbb E_P(v)^{1-p’}
\mathbb E_P(\sigma)^{-p}\\
&&\times\mathbb E_P(v \chi_{E(P)})^{p’-1}
\mathbb E_P(\sigma\chi_{E(P)})d\mu\\
&\leq&\eta^{p(p’-1)}[v]^{p’}_{A_{p}}\int_{P}
\mathbb E_P(\sigma)^{-p}
\mathbb E_P(\sigma\chi_{E(P)})
d\mu.\end{eqnarray*}
Because $E(P)$ is a subset of $P$ and $a^{{\mathcal{K}}_2(P)-1}\chi_{P}\leq\mathbb E_P(h)\chi_{P},$ we obtain that
\begin{eqnarray*}
\int_{E(P)}a^{p({\mathcal{K}}_2(P)-1)}vd\mu
&\leq&\eta^{p(p’-1)}[v]^{p’}_{A_{p}}\int_{P}\mathbb E_P(f\sigma)^p\mathbb E_P(\sigma)^{-p} \mathbb E_P(\chi_{E(P)}\sigma)d\mu\\
&=&\eta^{p(p’-1)}[v]^{p’}_{A_{p}}\int_{P}\mathbb E_P^{\sigma}(f)^p \mathbb E_P(\chi_{E(P)}\sigma)d\mu,
\end{eqnarray*}
where we have used $ \mathbb E_P(f\sigma)=\mathbb E_P^\sigma(f)\mathbb E_P(\sigma).$
Then
\begin{eqnarray*}
\int_{E(P)}a^{p({\mathcal{K}}_2(P)-1)}vd\mu
&\leq&\eta^{p(p’-1)}[v]^{p’}_{A_{p}}\int_{P}\mathbb E_P^{\sigma}(f)^p \mathbb E_P(\chi_{E(P)}\sigma)d\mu\\
&=&\eta^{p(p’-1)}[v]^{p’}_{A_{p}}\int_{P}\mathbb E_P^{\sigma}(f)^p \chi_{E(P)}\sigma d\mu\\
&\leq&\eta^{p(p’-1)}[v]^{p’}_{A_{p}}\int_{P}M^{\sigma}(f\chi_{P_0})^p \chi_{E(P)}\sigma d\mu\\
&=&\eta^{p(p’-1)}[v]^{p’}_{A_{p}}\int_{E(P)}M^{\sigma}(f\chi_{P_0})^p \sigma d\mu.
\end{eqnarray*}
It follows from \eqref{doob} and the boundedness of Doob maximal operator $M^{\sigma}$ that
\begin{eqnarray*}
\int_{P_0}{^*M_i}(f\sigma)^{p}vd\mu
&\leq&a^{2p}\eta^{p(p’-1)}[v]^{p’}_{A_{p}}\sum\limits_{P\in \mathcal{P}}\int_{E(P)}M^{\sigma}(f\chi_{P_0})^p \sigma d\mu\\
&\leq&a^{2p}\eta^{p(p'-1)}[v]^{p’}_{A_{p}}\sum\limits_{P\in \mathcal{P}}\int_{E(P)}M^{\sigma}(f\chi_{P_0})^p \sigma d\mu\\
&\leq&a^{2p}\eta^{p(p'-1)}(p')^{p}[v]^{p’}_{A_{p}}\int_{P_0}f^p\sigma d\mu,
\end{eqnarray*}
which implies \eqref{6}. Furthermore,
\begin{eqnarray*}
\int_{\Omega_0}{^*M_i}(f\sigma)^{p}vd\mu
&=&\sum\limits_{k\in \mathbb{Z}}\int_{\{a^{k-1}< E(f\sigma|\mathcal{F}_i)\leq a^k\}\cap\Omega_0}{^*M_i}
(f\sigma)^{p}vd\mu\\
&\leq&a^{2p}\eta^{p(p’-1)}(p’)^{p}[v]^{p’}_{A_{p}}\sum\limits_{k\in \mathbb{Z}}\int_{\{a^{k-1}< E(f\sigma|\mathcal{F}_i)\leq a^k\}\cap\Omega_0}f^p\sigma d\mu\\
&\leq&a^{2p}\eta^{p(p’-1)}(p’)^{p}[v]^{p’}_{A_{p}}\int_{\Omega_0}f^p\sigma d\mu.
\end{eqnarray*}
Noting that $(\Omega,\mathcal{F},\mu)$ is a $\sigma$-finite measure space, we obtain that
\begin{eqnarray*}
\big(\int_{\Omega}{^*M_i}(f\sigma)^{p}vd\mu\big)^{\frac{1}{p}}
&\leq&a^2\eta^{(p’-1)}p’[v]^{\frac{p’}{p}}_{A_{p}}\big(\int_{\Omega}f^p\sigma d\mu\big)^{\frac{1}{p}}.
\end{eqnarray*}
Because $^*M_i(\cdot)\uparrow M_i(\cdot)$ as $i\downarrow -\infty,$ then

\begin{equation}\label{2se}
\big(\int_{\Omega}{M}(f\sigma)^{p}vd\mu\big)^{\frac{1}{p}}
\leq a^2\eta^{(p’-1)}p’[v]^{\frac{p’}{p}}_{A_{p}}\big(\int_{\Omega}f^p\sigma d\mu\big)^{\frac{1}{p}}.
\end{equation}

{\bf Approach \eqref{Long}.}
For $f\in L^p(vd\mu),$ $b>1$ and $k\in \mathbb{Z},$ we define
stopping times $$\tau_k=\inf\{n:~|f_n|>b^k\}.$$ 
Then we denote
\begin{equation*}
A_{k,j}:=\{\tau_k<\infty\}\cap\{b^j<\mathbb{E}(\sigma|\mathcal{F}_ {\mathcal{F}_{\tau_k}})\leq b^{j+1}\}
\end{equation*}
and
\begin{equation*}
B_{k,j}:=\{\tau_k<\infty,\tau_{k+1}=\infty\}\cap\{b^j<\mathbb{E}(\sigma|\mathcal{F}_{ \mathcal{F}_{\tau_k}})\leq b^{j+1}\},~ j\in \mathbb{Z}.
\end{equation*}
It follows that $A_{k,j}\in \mathcal{F}_{\tau_k}, B_{k,j}\subseteq A_{k,j}$.
It is clear that $\{B_{k,j}\}_{k,j}$ is a family of disjoint sets and
\begin{equation*}
\{b^k<Mf\leq b^{k+1}\}=\{\tau_k<\infty,\tau_{k+1}=\infty\}=\bigcup\limits_{j\in \mathbb{Z}.} B_{k,j}, k\in \mathbb{Z}.
\end{equation*}
Following from
\begin{equation}\label{change}
\mathbb{E}(f|\mathcal{F}_{\tau_k})=\mathbb{E}^{\sigma}(f\sigma^{-1}|\mathcal{F}_{\tau_k}\big) \mathbb{E}(\sigma|\mathcal{F}_{\tau_k}),
\end{equation}
we have
\begin{eqnarray} b^{kp}&\leq&
                 \mathop{\hbox{ess inf}}\limits_{A_{k,j}}\mathbb{E}(f|\mathcal{F}_{\tau_k})^p\\
              &\leq&\mathop{\hbox{ess inf}}\limits_{A_{k,j}}\mathbb{E}^{\sigma}(f\sigma^{-1}|\mathcal{F}_{\tau_k}\big) ^p
                 \mathop{\hbox{ess sup}}\limits_{A_{k,j}}\mathbb{E}(\sigma|\mathcal{F}_{\tau_k})^p\\
              &\leq&b^p \mathop{\hbox{ess inf}}\limits_{A_{k,j}}\mathbb{E}^{\sigma}(f\sigma^{-1}|\mathcal{F}_{\tau_k})^p
                |B_{k,j}|_v^{-1}\int_{B_{k,j}}\mathbb{E}(\sigma|\mathcal{F}_{\tau_k})^p vd\mu.
\end{eqnarray}
Applying the $A_p$ condition
$$1\leq \mathbb{E}(v|\mathcal{F}_{\tau})\mathbb{E}(\sigma|\mathcal{F}_{\tau})^{p-1}\leq [v]_{A_p},~\forall \tau,$$
we have 
$$\mathbb{E}(\sigma|\mathcal{F}_{\tau_k})^p \leq [v]^{\frac{p}{p-1}}_{A_p}\mathbb{E}(v|\mathcal{F}_{\tau_k})^{-p’}=[v]^{\frac{p}{p-1}}_{A_p}\mathbb{E}^v(v^{-1}|\mathcal{F}_{\tau_k})^{p’}.$$

It follow that 
\begin{eqnarray}
\int_\Omega(Mf)^p v d\mu
   &=&\sum\limits_{k\in \mathbb{ Z}}\int_{\{b^k<Mf\leq b^{k+1}\}}(Mf)^p v d\mu\\
   &\leq&b^p\sum\limits_{k\in \mathbb{ Z}}\int_{\{b^k<Mf\leq b^{k+1}\}}b^{kp} v d\mu\\
   &=&b^p\sum\limits_{k\in \mathbb{ Z},j\in \mathbb{ Z}}\int_{B_{k,j}}b^{kp} v d\mu\\
   &\leq&b^{2p}[v]^{\frac{p}{p-1}}_{A_p}\sum\limits_{k\in\mathbb{ Z},j\in
        \mathbb{ Z}}\mathop{\hbox{ess inf}}\limits_{A_{k,j}}\mathbb{E}^{\sigma}(f\sigma^{-1}|\mathcal{F}_{\tau_k})^p
                \int_{B_{k,j}}\mathbb{E}^v(v^{-1}|\mathcal{F}_{\tau_k})^{p’} vd\mu.
\end{eqnarray} 

Letting 
$X:=\mathbb{Z}^2$ and \begin{equation*}
\vartheta(k,j):=\int_{B_{k,j}}\mathbb{E}^v(v^{-1}|\mathcal{F}_{\tau_k})^{p’} vd\mu,
\end{equation*} we have
that $\vartheta$ is a measure on $X.$ For $f\in L^p(vd\mu)$ and $\lambda>0,$ we denote
\begin{eqnarray}
Tf(k,j)&:=&\mathop{\hbox{ess inf}}\limits_{A_{k,j}}\mathbb{E}^{\sigma}(f\sigma^{-1}|\mathcal{F}_{\tau_k})^p,\\
 \mathbb{E}_\lambda&:=&\Big\{(k,j):\mathop{\hbox{ess inf}}\limits_{A_{k,j}}\mathbb{E}^{\sigma}(f\sigma^{-1}|\mathcal{F}_{\tau_k})^p>\lambda\Big\},\\
G_\lambda&:=&\bigcup\limits_{(k,j)\in \mathbb{E}_\lambda}A_{k,j} .
\end{eqnarray}
 It follows that 
\begin{eqnarray} |\{Tf>\lambda\}|_\vartheta
&=&\sum\limits_{(k,j)
        \in E_\lambda}\int_{B_{k,j}}\mathbb{E}^v(v^{-1}|\mathcal{F}_{\tau_k})^{p’} vd\mu\\
&\leq&\sum\limits_{(k,j)\in \mathbb{E}_\lambda}\int_{B_{k,j}}\mathbb{E}^v(v^{-1}\chi_{G_\lambda}|\mathcal{F}_{\tau_k})^{p’} vd\mu\\
&\leq& \int_{G_\lambda}\big(M^v(v^{-1}\chi_{G_\lambda})\big)^{p’}vd\mu. 
\end{eqnarray} 
For $\tau=\inf\Big\{n:~\mathbb{E}^{\sigma}(f\sigma^{-1}|\mathcal{F}_n)^p>\lambda\Big\},$ we obtain
$G_\lambda\subseteq\Big\{M^{\sigma}(f\sigma^{-1})^p>\lambda\Big\}=\{\tau<\infty\}.$
In view of the boundedness of Doob maximal operator $M^v,$ we get that
 \begin{eqnarray} |\{Tf>\lambda\}|_\vartheta
&\leq&\int_{G_\lambda}\big(M^{v}(v^{-1}
      \chi_{G_\lambda})\big)^{p’}vd\mu\\
&\leq& \int_{\{\tau<\infty\}}\big(M^{v}(v^{-1}
      \chi_{\{\tau<\infty\}})\big)^{p’}vd\mu \\
&\leq& p^{p’}|\{\tau<\infty\}|_{\sigma}\\
&=&p^{p’}|\{M^{\sigma}(f\sigma^{-1})^p>\lambda\}|_{\sigma}.
 \end{eqnarray}
Therefore
\begin{eqnarray}\label{share2}~~~~~~\int_{\Omega}(Mf)^p vd\mu
&\leq&b^{2p}[v]^{\frac{p}{p-1}}_{A_p}\int_XTfd\vartheta=b^{2p}[v]^{\frac{p}{p-1}}_{A_p}\int_0^\infty|\{Tf>\lambda\}|_\vartheta
              d\lambda\\
&\leq& b^{2p}p^{p’}[v]^{\frac{p}{p-1}}_{A_p}\int_0^{\infty}|\{M^{\sigma}(f\sigma^{-1})^p>\lambda\}|_{\sigma}
              d\lambda\nonumber  \\
&=&b^{2p}p^{p’}[v]^{\frac{p}{p-1}}_{A_p}\int_\Omega M^{\sigma}(f\sigma^{-1})^p\sigma
             d\mu.
\end{eqnarray} 
Using the boundedness of Doob maximal operator $M^{\sigma},$ we conclude that
\begin{equation}\label{share3}\int_{\Omega}(Mf)^p vd\mu
\leq b^{2p}p^{p’}{p’}^p[v]^{\frac{p}{p-1}}_{A_p}\int_\Omega|f|^pvd\mu.
\end{equation}

Taking limit as $b\rightarrow 1+$ in \eqref{share3}, we have
\begin{equation}\label{share4}\lVert M  f\rVert _{L ^{p} (v)} 
\leq p’p^{^{\frac{1}{p-1}}}[v]^{\frac{1}{p-1}}_{A_p} \lVert f\rVert _{L ^{p} (v)}.
\end{equation}

\end{proof}

\section{Comparison of $p^{\frac{1}{p-1}}$ and $a^2\eta^{(p’-1)}$}\label{Compar}

We compare $p^{\frac{1}{p-1}}$ with $a^{2}\eta^{(p’-1)}$ in this section, 
where $a>1$ and $\eta=\frac{a}{a-1}$ are the constants in the construction of principal sets (Appendix \ref{Appe}).
 We split our comparison into two theorems, Theorem \ref{mini} and Theorem \ref{compar-lim}.

\begin{theorem}\label{mini}
For $1<p<+\infty,$ let $\varphi(a)=a^{2}\eta^{(p’-1)}.$ Then we have 
\begin{equation}
\min\limits_{a>1}\varphi(a)=\varphi(\frac{2p-1}{2p-2}).
\end{equation}
\end{theorem}

\begin{proof} We deal with $\ln\varphi(a).$ Then $$\ln\varphi(a)=2\ln a+\frac{1}{p-1}\ln{\frac{a}{a-1}}.$$

It is easy to check $\lim\limits_{a \rightarrow 1+}\ln \varphi(a)=\lim\limits_{a\rightarrow+\infty}\ln\varphi(a)=+\infty.$ We have
$$\big(\ln\varphi(a) \big)’=\frac{2}{a}+\frac{1}{a(p-1)}-\frac{1}{(a-1)(p-1)}.$$
It is clear that the unique $a_0=:\frac{2p-1}{2p-2}$ solves equation $\big(\ln\varphi(a) \big)’=0$ and $a_0=\frac{2p-1}{2p-2}>1.$

Thus
$$
\min\limits_{a>1}\varphi(a)=\varphi(\frac{2p-1}{2p-2})=(\frac{2p-1}{2p-2})^2(2p-1)^{\frac{1}{p-1}}.
$$

\end{proof}

It follows from Theorem \ref{mini} that the minimum of $\varphi(a)$ is a function of $p.$ Then we denote the 
minimum $(\frac{2p-1}{2p-2})^2(2p-1)^{\frac{1}{p-1}}$ and the constant $p^{\frac{1}{p-1}}$  by $\phi(p)$ and $\psi(p),$ respectively. Because of $\frac{2p-1}{2p-2}>1$ and $2p-1>p,$
we have $\phi(p)\geq \psi(p).$ Now we study limits of $\phi(p)$ and $\psi(p)$ 
in the following Theorem   
\ref{compar-lim}.

\begin{theorem}\label{compar-lim} Let $\phi$ and $\psi$ as above.
Then
\begin{equation}\label{limit1}
\lim\limits_{p\rightarrow 1+}\phi(p)=+\infty,~\lim\limits_{p\rightarrow 1+}\psi(p)=e
\end{equation}
and 
\begin{equation}\label{limit2}
\lim\limits_{p\rightarrow +\infty}\phi(p)=\lim\limits_{p\rightarrow +\infty}\psi(p)=1.
\end{equation}
Moreover
\begin{equation}\label{frac}
\lim\limits_{p\rightarrow+\infty}\frac{\ln \phi(p)}{\ln \psi(p)}=1.
\end{equation}
\end{theorem}

\begin{proof} Because
\begin{equation}
\lim\limits_{p\rightarrow 1+}\ln\phi(p)=\lim\limits_{p\rightarrow 1+}2\ln(\frac{2p-1}{2p-2})+\lim\limits_{p\rightarrow 1+}\frac{1}{p-1}\ln(2p-1)=+\infty,
\end{equation}
and 
\begin{equation}
\lim\limits_{p\rightarrow +\infty}\ln\phi(p)=\lim\limits_{p\rightarrow +\infty}2\ln(\frac{2p-1}{2p-2})+\lim\limits_{p\rightarrow +\infty}\frac{1}{p-1}\ln(2p-1)=0,
\end{equation}
we have $\lim\limits_{p\rightarrow 1+}\phi(p)=+\infty$
and $\lim\limits_{p\rightarrow +\infty}\phi(p)=1,$
respectively.

Similarly, we get $\lim\limits_{p\rightarrow 1+}\psi(p)=e$ and $\lim\limits_{p\rightarrow +\infty}\psi(p)=1.$

Finally, we obtain
\begin{eqnarray}
\lim\limits_{p\rightarrow+\infty}\frac{\ln \phi(p)}{\ln \psi(p)}
&=&\lim\limits_{p\rightarrow+\infty}\frac{2\ln(\frac{2p-1}{2p-2})+\frac{1}{p-1}\ln(2p-1)}{\frac{1}{p-1}\ln p}\\
&=&\lim\limits_{p\rightarrow+\infty}\frac{2(p-1)\ln(\frac{2p-1}{2p-2})+\ln(2p-1)}{\ln p}\\
&=&\lim\limits_{p\rightarrow+\infty}\frac{2(p-1)\ln(\frac{2p-1}{2p-2})}{\ln p}+\lim\limits_{p\rightarrow+\infty}\frac{\ln(2p-1)}{\ln p}\\
&=&0+1=1.
\end{eqnarray}

\end{proof}

\begin{remark}\label{compar} We give further properties of $\phi(p)$ and $\psi(p).$

\begin{enumerate}
\item We claim that the function $\phi(p)$ is decreasing on $(1, +\infty).$ Writing $\phi_1(p)=(\frac{2p-1}{2p-2})^2$ and $\phi_2(p)=(2p-1)^{\frac{1}{p-1}},$ 
we  will show that $\phi_1(p)$ and $\phi_2(p)$ are both decreasing on $(1, +\infty).$  Combining this with $0<\phi_1(p)$ and $0<\phi_2(p),$ we obtain that $\phi(p)$ is decreasing on $(1, +\infty).$ 
We now check that $\phi_1(p)$ and $\phi_2(p)$ are both decreasing.

For $\phi_1(p)$ with $p\in(1,+\infty),$ 
it is clear that 
\begin{eqnarray}
\phi_1(p)=(\frac{2p-1}{2p-2})^2=(1+\frac{1}{2p-2})^2.
\end{eqnarray}
Thus $\phi_1(p)$ is decreasing on $(1, +\infty).$ 

For $\phi_2(p)$ with $p\in(1,+\infty),$ to consider  $$\ln\phi_2(p)=\frac{1}{p-1}\ln(2p-1).$$ 
It is clear that 
\begin{eqnarray}
\big(\ln\phi_2(p)\big)’&=&\frac{1}{{(p-1)}^2}\Big((\frac{2}{2p-1})(p-1)-\ln(2p-1)\Big)\\
&=&\frac{1}{{(p-1)}^2}\Big(\frac{2(p-1)}{2p-1}-\ln(2p-1)\Big).\end{eqnarray}
Using the Mean Value Theorem, we have 
$$\ln(2p-1)=\ln(2p-1)-\ln 1=\frac{1}{\xi}(2p-1-1)=\frac{1}{\xi}\big(2(p-1)\big),$$
where $\xi\in (1,2p-1).$
It follows that \begin{eqnarray}
\ln(2p-1)>\frac{2(p-1)}{2p-1},
\end{eqnarray} which implies $\big(\ln\phi_2(p)\big)’<0.$
Thus $\phi_2(p)$ is decreasing on $(1, +\infty).$ 

\item We claim that the function $\psi(p)$ is decreasing on $(1, +\infty).$ It suffices to show that $\psi’(p)<0.$ We have $$\psi’(p)=\frac{\psi(p)}{(p-1)^2}(1-\frac{1}{p}+\ln\frac{1}{p}).$$
It is clear that $\psi’(p)<0$ if and only if $1-\frac{1}{p}+\ln\frac{1}{p}<0.$ Let $s(t)=1-t+\ln t$ with $t\in (0,1].$ Because of $s’(t)=\frac{1}{t}-1>0$ on $(0,1),$ the function $s(t)$ is strictly increasing on
 $(0,1].$ It follows from $s(1)=0$ that $s(t)<0$ on $(0,1).$ That is $1-\frac{1}{p}+\ln\frac{1}{p}<0$ with $p>1.$ Thus $\psi(p)$ is decreasing on $(1, +\infty).$
\end{enumerate}

At the end of Section \ref{Compar},  we check our work with graphing device in Figure $\ref{Fig1}.$

\begin{figure}[htbp]
\begin{minipage}[t]{120mm}
\vspace {2mm}
\centering\includegraphics[width=8cm]{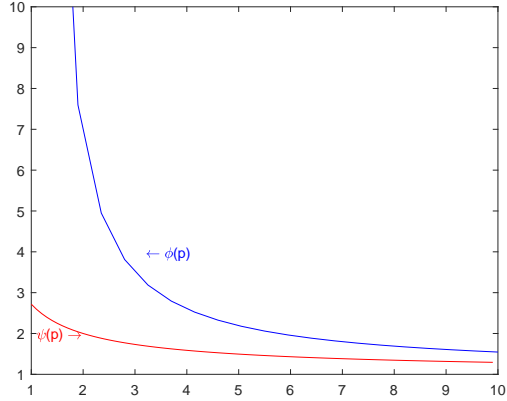}
\caption{Computer confirmation of $\phi(p)$ and $\psi(p)$}
\label{Fig1}
\end{minipage}
\end{figure}

\end{remark}

\appendix
\section{Construction of Principal Sets}\label{Appe}

The construction of principal sets first appeared in Tanaka and Terasawa \cite{MR3004953},
and Chen, Zhu, Zuo and Jiao \cite{MR4244905, MR4125846} found the conditional sparsity of the construction, which is new and useful. We will use  the construction
of principal sets. Because we keep track the constants of the conditional sparsity,  we will give the modifications
in the construction of principal sets in this Appendix. 

For $i\in \mathbb Z,$ $h\in \mathcal{L}^+,$ $a>1$
and $k\in \mathbb Z,$ stopping times are defined by$$
\tau:=\inf\{j\geq i:
~\mathbb E(h|\mathcal{F}_j)>a^{k+1}\}.$$
Let \begin{equation}\label{P0}
P_0:=\{a^{k-1}< \mathbb E(h|\mathcal{F}_i)\leq a^k\}\cap\Omega_0,\end{equation}
where $\Omega_0\in \mathcal{F}^0_i,$
then $P_0\in \mathcal{F}^0_i.$
We denote $\mathcal{K}_1(P_0):=i$ and $\mathcal{K}_2(P_0):=k.$ Then we define
$\mathcal{P}_1:=\{P_0\},$  which is the first generation $\mathcal{P}_1$. Now we show how to define the second one. Let
$$
\tau_{P_0}:=\tau\chi_{P_0}+\infty\chi_{P^c_0},$$
where $P^c_0=\Omega\setminus P_0.$
Let $P$ be a subset of $ P_0$ with $\mu(P)>0.$ 
If  there is $i<j$ and $k+1<j$ such that
\begin{eqnarray}
P&=&\{a^{l-1}< \mathbb E(h|\mathcal{F}_j)
          \leq a^l\}\cap\{\tau_{P_0}=j\}\cap P_0\\
&=&\{a^{l-1}< \mathbb E(h|\mathcal{F}_j)
          \leq a^l\}\cap\{\tau=j\}\cap P_0,\end{eqnarray}
we say that $P$ is a principal set of $P_0.$
We denote $\mathcal{K}_1(P):=j$ and $\mathcal{K}_2(P):=l.$ Letting $\mathcal{P}(P_0)$ be the family of the above principal sets of $P_0,$  we say that 
$\mathcal{P}_2:=\mathcal{P}(P_0)$ is the second generation.

Following \cite[P.804]{MR4125846}, we have
\begin{equation}\label{constant}
\mu(P_0)\leq\frac{a}{a-1}\mu\big(E(P_0)\big)=:\eta \mu\big(E(P_0)\big)
\end{equation}
where 
$$
E(P_0):=P_0\cap\{\tau_{P_0}=\infty\}=P_0\cap\{\tau=\infty\}=P_0\backslash\bigcup\limits_{P\in \mathcal{P}(P_0)}P.
$$
Furthermore,  we have
 $\chi_{P_0}\leq \eta\mathbb E_i(\chi_{E(P_0)})\chi_{P_0},$ 
which is called {\bf the conditional sparsity of principal sets with $\eta$}(see \cite{MR4244905, MR4125846}).

Proceeding inductively, we obtain the next generalizations
$$
\mathcal{P}_{n+1}:=\bigcup\limits_{P\in \mathcal{P}_{n}}\mathcal{P}(P).
$$
Let $$
\mathcal{P}:=\bigcup\limits_{n=1}^{\infty}\mathcal{P}_{n},
$$
then the collection of principal sets $\mathcal{P}$ satisfies the following properties:
\begin{enumerate}
  \item The sets $E(P)$ where $P\in \mathcal{P},$ are disjoint and $P_0=\bigcup\limits_{P\in \mathcal{P}}E(P);$
  \item $P\in\mathcal{F}_{{\mathcal{K}}_1(P)};$
  \item  \label{property3}$\chi_{P}\leq \eta\mathbb E(\chi_{E(P)}|\mathcal{F}_{{\mathcal{K}}_1(P)})\chi_{P};$
  \item  $a^{{\mathcal{K}}_2(P)-1}<\mathbb E(h|\mathcal{F}_{{\mathcal{K}}_1(P)})\leq a^{{\mathcal{K}}_2(P)}$ on $P;$
  \item  $\sup\limits_{j\geq i}\mathbb E_j(h\chi_P)
                 \leq a^{{\mathcal{K}}_2(P)+1}$ on $E(P);$
  \item  $\chi_{\{\mathcal{K}_1(P)\leq j<\tau(P)\}}\mathbb E_j(h)\leq a^{{\mathcal{K}}_2(P)+1}.$
\end{enumerate}
where $\eta=a/(a-1).$

Now, we represent the tailed Doob maximal operator by the principal sets, which is the following lemma.
\begin{lemma}\label{repre}Let $h\in \mathcal{L}^+,$ $a>1$ and $i\in \mathbb{Z}.$
For $k\in \mathbb{Z}$ and $\Omega_0\in \mathcal{F}^0_i,$ we let $$P_0:=\{a^{k-1}< \mathbb E(h|\mathcal{F}_i)\leq a^k\}\cap\Omega_0.$$
If $\mu(P_0)>0,$ then
\begin{eqnarray*}
{^*M_i}(h)\chi_{P_0}
&=&{^*M_i}(h\chi_{P_0})\chi_{P_0}\\
&=&\sum\limits_{P\in \mathcal{P}}{^*M_i}(h\chi_{P_0})\chi_{E(P)}\\
&\leq&a^2\sum\limits_{P\in \mathcal{P}}a^{({\mathcal{K}}_2(P)-1)}\chi_{E(P)}.
\end{eqnarray*}
\end{lemma}

\bibliographystyle{alpha,amsplain}	


\end{document}